\documentclass[12pt]{article}
\usepackage[utf8]{inputenc}
\usepackage[a4paper,inner=2.5cm,outer=2.5cm,top=2.5cm,bottom=2.5cm]{geometry}
\usepackage{amsmath,amssymb,amsthm}

\usepackage{hyperref}
\hypersetup{colorlinks=true,linkcolor=blue,citecolor=teal,
filecolor=magenta,urlcolor=cyan}

\usepackage{tikz}
\usetikzlibrary{calc,decorations.markings,arrows.meta}
\tikzset{baseline={([yshift=-.5ex]current bounding box.center)}}
\tikzstyle directed=[blue,postaction={decorate,decoration={markings, mark= at position .55 with {\arrow{stealth}}}}]
\tikzstyle{v} = [draw,black,fill,circle, inner sep=1pt]
\tikzstyle{->}=[-{Stealth[]}]
\tikzstyle{<->}=[{Stealth[]}-{Stealth[]}]
\tikzstyle{<-}=[{Stealth[]}-]

\usepackage{xcolor}

\newtheorem{theorem}{Theorem}[section]
\newtheorem{lemma}[theorem]{Lemma}
\newtheorem{corollary}[theorem]{Corollary}
\newtheorem{proposition}[theorem]{Proposition}

\theoremstyle{definition}
\newtheorem{definition}[theorem]{Definition}
\newtheorem{remark}[theorem]{Remark}
\newtheorem{example}[theorem]{Example}

\def\S{{\mathbb S}}

\newcommand{\cA}{{\mathcal{A}}}
\newcommand{\cB}{{\mathcal{B}}}

\newcommand{\fgl}{{\mathfrak{gl}}}

\newcommand{\fso}{{\mathfrak{so}}}

\title{Chromatic polynomial and the $\fso$ weight system}
\author{S. Lando\thanks{Higher School of Economics, Skolkovo Institute of Science and Technology
}, Z. Yang\thanks{Beijing Institute of Mathematical Sciences and Applications}}
\date{2024}

\begin{document}

\maketitle

\begin{abstract}    
    In a recent paper by M.~Kazarian and the second author, a recurrence for the Lie algebras~$\fso(N)$ weight systems has been suggested; the recurrence allows one to construct the universal $\fso$ weight system. 
    The construction is based on an extension of the $\fso$ weight systems to permutations. 
    Another recent paper, by M.~Kazarian, N.~Kodaneva, and the first author,
    shows that under the substitution $C_m=xN^{m-1}, m=1,2,\dots,$ 
    for the Casimir elements~$C_m$, 
    the leading term in $N$ of the value of the universal $\fgl$ weight system becomes the chromatic polynomial of the intersection graph of the chord diagram. 
    In the present paper, we establish a similar result for the universal $\fso$ weight system.
    That is, we show that the leading term of the universal $\fso$ weight system also becomes the chromatic polynomial under a specific substitution.
\end{abstract}

\begin{tabular}{rl}

$\fgl(N)$ & Lie algebra of all $N\times N$ matrices \\
$E_{ij}$ & matrix unit, $1\le i,j\le N$\\
$w_\fgl$ & weight system associated to the series $\fgl$ of Lie algebras\\
$U\fgl(\infty)$ & universal enveloping algebra of the limit Lie algebra $\fgl(\infty)$\\
$ZU\fgl(N)$ & center of $U\fgl(N)$\\
$\fso(N)$ & Lie algebra of special orthogonal $N\times N$ matrices \\
$w_\fso$ & weight system associated to the series $\fso$ of Lie algebras\\
$U\fso(\infty)$ & universal enveloping algebra of the limit Lie algebra $\fso(\infty)$\\
$ZU\fso(N)$ & center of $U\fso(N)$\\
$C_1,C_2,\dots$ & Casimir elements generating the center of~$U\fgl(\infty)$ \\
$S_2,S_4,\dots$ & Casimir elements generating the center of~$U\fso(\infty)$ \\
$\alpha\in\S_m$ & permutation, which determines the hyperedges of the hypermap\\
$\S_m$ & group of permutations of $m$ elements\\
$f(\alpha)$ & number of faces of~$\alpha$\\
$c(\alpha)$ & number of cycles in~$\alpha$\\
$V(\alpha)$ & set of cycles of $\alpha$\\
$\ell(\alpha)$ & the sum of the lengths of all negative cycles in~$\alpha$\\
$d(\alpha)$ & the number of length two orbits ('chords') in~$\alpha$\\
$v_1, v_2, \ldots\in V(\alpha)$  & cycles of $\alpha$\\
$a(\alpha)$ & number of ascents of $\alpha$\\
$\alpha|_U$ & restriction of~$\alpha$ to $U\subset\{1,2,\ldots m\}$\\

$\sigma=(1,2,\dots,m)\in\S_m$ & standard long cycle, the hypervertex of a hypermap\\

$\cA_m$ & the vector space spanned by equivalence classes of permutations\\ &of $m$ elements\\
$[\alpha]$ & equivalence class of $\alpha$\\
$\cA$ & the rotational Hopf algebra of permutations\\
$\alpha_1 \# \alpha_2$ & concatenation product of permutations $\alpha_1$ and $\alpha_2$\\
$\mu(\alpha)$ & coproduct of $\alpha$\\
$\bar\alpha$ & canonical form of $\alpha$\\

$\cA_m^+$ & the vector space spanned by equivalence classes\\& of positive permutations of $m$ elements\\
$\cA^+$ & the rotational Hopf algebra of positive permutations\\
$\cA^-$ & the rotational Hopf algebra of negetive permutations\\
$\cA^\pm$ & the rotational Hopf algebra of monotone permutations\\

$\cB_m$ & the vector space spanned by equivalence classes\\& of involutions of $m$ elements without fixed points\\
$\cB$ & the rotational Hopf algebra of chord diagrams\\



$G$ & simple graph\\
$K_n$ & complete graph with $n$ vertices\\
$V(G)$ & the set of vertices of~$G$\\
$G|_U$  & the subgraph of $G$ induced by~$U\subset V(G)$\\

$\chi_G(x)$ & chromatic polynomial of~$G$\\

$X(\alpha)$ & chromatic substitution of $\fgl$ weight system\\
$X_0(\alpha)$ & top coefficient of $X(\alpha)$\\
$Y(\alpha)$ & chromatic substitution of $\fso$ weight system\\
$\hat{Y}(\alpha)$ & simplified prechromatic substitution of $\fso$ weight system\\
$Y_0(\alpha)$ & top coefficient of $Y(\alpha)$\\

\end{tabular}

\newpage
\section{Introduction}
Weight systems are functions on chord diagrams
satisfying the so-called Vassiliev $4$-term relations.
They are closely related to finite-type knot invariants;
see~\cite{V90,K93}. 

Certain weight systems can be derived
from graph invariants; see a recent account in~\cite{KL22}.
One of the first examples of such an invariant was the chromatic
polynomial~\cite{CDL1}.

Another main source of weight systems are Lie algebras,
the construction due to D.~Bar-Natan~\cite{BN95} 
and M.~Kontsevich~\cite{K93}.
In the papers~\cite{ZY22,KL22}, an idea due to M.~Kazarian was used to define a
universal $\fgl$ weight system, which takes values in
the ring $\mathbb{C}[N,C_1,C_2,C_3,\dots]$ of polynomials in
infinitely many variables. The construction makes use of an extension
of the $\fgl(N)$ weight systems, $N=1,2,\dots$, to permutations,
which coincides with the
$\fgl(N)$ weight system on chord diagrams (the latter being understood as involutions
without fixed points). 

In paper~\cite{KKL24}, it is shown that the leading term in $N$ 
of the value of the universal $\fgl$ weight system
under the substitution $C_k=xN^{k-1}$ becomes the chromatic polynomial
of the intersection graph of the chord diagram. Since the substitution
$C_k=N^{k-1}$, $k=1,2,\dots$, is the specification of the $\fgl(N)$ Lie
algebra weight systems in the standard representation, the substitution $C_k=xN^{k-1}$
can be considered as a perturbation of the latter.
In the present paper, 
we show that the leading term of the universal $\fso$ 
weight system under a specific substitution is also the chromatic polynomial.
Like the substitution in~\cite{KKL24}, ours is a perturbation
of the standard representation of the Lie algebras $\fso(N)$.

The result of~\cite{KKL24} is valid not only for chord diagrams, but also for
more general permutations, which are called positive permutations.
In the case of the universal $\fso$ weight system, we show that a similar assertion is
valid not only for positive permutations, but for a wider class of monotone permutations.

\section{Construction of the universal $\fso$ weight system}

\subsection{Definition of $w_\fso$}

Following~\cite{KY23}, we define in this section a multiplicative weight system $w_\fso$ taking values in the ring of polynomials in the generators $S_0,S_2,S_4,S_6,\dots$ labeled by even non-negative integers. Similarly to the case of the $w_\fgl$ weight system, the definition of $w_\fso$ 
proceeds by extending it to the set of permutations (on any number of permuted elements). This extension is defined by a set of relations that are close to those of the $w_\fgl$ case. The defining relations presented below are motivated by the requirement that for the Lie algebra $\fso(N)$ (with $N$ of any parity) the specialization of $w_\fso$ taking $S_k$ to the corresponding  Casimir element is given by
 \begin{equation}\label{eq:wsonsigma}
	w_{\fso(N)}(\alpha)=\sum_{i_1,\cdots,i_m=1}^N X_{i_1i_{\alpha(1)}}X_{i_2i_{\alpha(2)}}\cdots X_{i_mi_{\alpha(m)}}\in U\fso(N),
\end{equation}
where $X_{ij}$ are the standard generators of $\fso(N)$:
$$
X_{ij}=E_{ij}-E_{\bar j\bar i},\qquad \bar i=N+1-i.
$$
Below, we will call the elements $\{1,2,\dots,m\}$ permutated by a permutation~$\alpha$
the \emph{legs} of the permutation.

 
Before formulating defining relations for $w_\fso$, we introduce some notation used in these relations. For $w_\fgl$, the defining relations
are naturally expressed in terms of digraphs of permutations~\cite{ZY22}.
Such a digraph just shows where each permuted element is being taken.

The invariant $w_\fso$ constructed below possesses an additional symmetry that does not hold for the $\fgl$ case:
\begin{quote}
\emph{assume that the permutation $\alpha'$ is obtained from $\alpha$ by  inverting one of its disjoint cycles. In this case, the values $w_\fso(\alpha)$ and $w_\fso(\alpha')$ differ by the sign factor $(-1)^\ell$ where $\ell$ is the length of the cycle.}
    \end{quote}

This symmetry leads to the following convention. Along with the digraphs of permutations, we consider more general graphs, which we call extended permutation graphs.

\begin{definition} An \emph{extended permutation graph} is a graph with the following properties:
\begin{itemize}
\item the set of vertices of the graph is linearly ordered, which is depicted by placing them on an additional oriented line in the order compatible with the orientation of that line;
\item each vertex has valency~$2$, in particular, the number of edges is equal to the number of vertices;
\item for each half-edge it is specified whether it is a \emph{head} (marked with an arrow) or a \emph{tail}. For two half-edges adjacent to every vertex, one of them is a tail, and the other is a head. However, we allow the edges of the graph to have two heads, or two tails, or a head and a tail.
\end{itemize}
\end{definition}

In an ordinary permutation graph, each edge has one tail and one head. 
We extend $w_\fso$ to 
extended permutation graphs by the following additional relation: 
a change of the tail and the head for the two
half-edges adjacent to any vertex results in the multiplication 
of the value of the invariant by $-1$:
$$
w_\fso\left(~
\begin{tikzpicture}[scale=0.8]
\draw (0.9,0) node[v] (V) {};
\draw [->] (0,0) -- (1.8,0);
\draw[->,blue] (1.8,1) .. controls +(-0.6,-0.5) .. (V);
\draw[blue] (V) .. controls +(-0.3,0.5) .. (0,1);
\end{tikzpicture}
~\right)=-w_\fso\left(~
\begin{tikzpicture}[scale=0.8]
\draw (0.9,0) node[v] (V) {};
\draw [->] (0,0) -- (1.8,0);
\draw[blue] (1.8,1) .. controls +(-0.6,-0.5) .. (V);
\draw[->,blue] (0,1) .. controls (0.6,0.5) .. (V);
\end{tikzpicture}
~\right).
$$
By applying this transformation several times, every extended permutation graph can be reduced to a usual permutation graph (characterized by the additional property that every edge has one head and one tail). This permutation graph is not unique
(two such graphs differ by orientation of some cycles), but the symmetry of $w_\fso$ formulated above implies that this ambiguity does not affect the extension of $w_\fso$.

\begin{definition}\label{def:so}
The universal polynomial invariant $w_\fso$ is the function on the set of permutations of any number of elements (or equivalently, on the set of permutation graphs) taking values in the ring of polynomials in the generators $S_0,S_2,S_4,\dots$ and defined by the following set of relations (axioms):
\begin{itemize}
\item $w_{\fso}$ is multiplicative with respect to the concatenation of the permutation graphs, $w_\fso(\alpha_1\#\alpha_2)=w_\fso(\alpha_1)w_\fso(\alpha_2)$. 
As a corollary, for the empty graph (with no vertices) the value of $w_{\fso}$ is equal to $1$;
\item A change of orientation of any cycle of length~$\ell$ in the graph results in multiplication of the value of the invariant~$w_\fso$ by~$(-1)^\ell$.
\item for a cyclic permutation of even number of elements {\rm(}with the cyclic order on the set of permuted elements  compatible with the permutation{\rm)} $1\mapsto2\mapsto\cdots\mapsto m\mapsto1$,
 the value of $w_{\fso}$ is the standard generator,
$$w_\fso\left(
  \begin{tikzpicture}
\draw (0,0) node[v] (V1) {}  node[above] {$\scriptsize 1$} (0.6,0) node[v] (V2) {}  node[above] {$\scriptsize 2$} (1.2,0) node[v] (V3) {} (2.4,0) node[v] (V4) {} (3,0) node[v] (V5) {}  node[above] {$\scriptsize m$}
(1.8,0) node {$\cdots$};
\draw[->] (-0.4,0) -- (3.6,0);
 \draw[->,blue] (V1) .. controls +(.3,.2) and +(-.3,.2) ..  (V2);
 \draw[->,blue] (V2) .. controls +(.3,.2) and +(-.3,.2) .. (V3);
 \draw[->,blue] (V4) .. controls +(.3,.2) and +(-.3,.2) ..  (V5);
 \draw[->,blue] (V5) .. controls +(-.5,-.5) and +(.5,-.5) .. (V1);
 \end{tikzpicture}
 \right)=S_m,\qquad m\text{ is even}.
 $$
\item {\rm(}\textbf{Recurrence Rule}{\rm)} For the graph of an arbitrary permutation $\alpha$,
and for any pair of its vertices labelled by consecutive integers $r,r+1$, we have for the values of the $w_\fso$ weight system
\begin{multline}\label{eq:rec-so1}
w_\fso\left(~
\begin{tikzpicture}[scale=0.8]
\draw (0.6,0) node[v] (V1) {} (1.2,0) node[v] (V2) {};
\draw  (0.5,-.25) node {\scriptsize$r$}  
 (1.3,-.25) node {\scriptsize$r{+}1$};
\draw [->] (0,0) -- (1.8,0);
\draw[->,blue] (0,0.7) .. controls (0.6,0.5) .. (V2);
\draw[-,blue] (V2) .. controls (0.6,0.9) .. (0,1.4);
\draw[->,blue] (1.8,1.4) .. controls (1.2,0.9) .. (V1);
\draw[-,blue] (V1) .. controls (1.2,0.5) .. (1.8,0.7);
\end{tikzpicture}
 ~\right)=w_\fso\left(~
\begin{tikzpicture}[scale=0.8]
\draw  (0.7,0) node[v] (V1) {} (1.1,0) node[v] (V2) {};
\draw [->] (0,0) -- (1.8,0);
\draw[->,blue] (0,0.7) .. controls (0.4,0.4) .. (V1);
\draw[-,blue] (V1) .. controls (0.4,0.9) .. (0,1.4);
\draw[->,blue] (1.8,1.4) .. controls (1.4,0.9) .. (V2);
\draw[-,blue] (V2) .. controls (1.4,0.4) .. (1.8,0.7);
\end{tikzpicture}
 ~\right)+w_\fso\left(~
\begin{tikzpicture}[scale=0.8]
\draw (0.9,0) node[v] (V) {};
\draw [->] (0,0) -- (1.8,0);
\draw[-,blue] (0,0.7) .. controls +(0.5,-0.5) and +(-0.5,-0.5) .. (1.8,0.7);
\draw[->,blue] (1.8,1.4) .. controls +(-0.6,-0.7) .. (V);
\draw[-,blue] (V) .. controls +(-0.3,0.7) .. (0,1.4);
\end{tikzpicture}
 ~\right)-w_\fso\left(~
\begin{tikzpicture}[scale=0.8]
\draw (0.9,0) node[v] (V) {};
\draw [->] (0,0) -- (1.8,0);
\draw[->,blue] (0,0.7) .. controls (0.4,0.5) .. (V);
\draw[-,blue] (V) .. controls (1.4,0.5) .. (1.8,0.7);
\draw[-,blue] (1.8,1.4) .. controls (0.9,0.2) .. (0,1.4);
\end{tikzpicture}
 ~\right)
 \\
+w_\fso\left(~
\begin{tikzpicture}[scale=0.8]
\draw (0.9,0) node[v] (V) {};
\draw [->] (0,0) -- (1.8,0);
\draw[-,blue] (0,0.7) .. controls (0.6,0.5) and +(-0.5,-0.7) .. (1.8,1.4);
\draw[<-,blue] (V) .. controls (0.6,0.9) .. (0,1.4);
\draw[-,blue] (V) .. controls (1.2,0.5) .. (1.8,0.7);
\end{tikzpicture}
 ~\right)-w_\fso\left(~
\begin{tikzpicture}[scale=0.8]
\draw (0.9,0) node[v] (V) {};
\draw [-] (0,0) -- (1.8,0);
\draw[-,blue] (1.8,0.7) .. controls +(-0.6,-0.4) and (0.6,0.9) .. (0,1.4);
\draw[->,blue] (1.8,1.4) .. controls (1.2,0.7) .. (V);
\draw[-,blue] (V) .. controls (0.6,0.5) .. (0,0.7);
\end{tikzpicture}
 ~\right)\end{multline}
\end{itemize}
The first three graphs on the right are ordinary permutation graphs, which
are the same as in the defining relation for the $w_\fgl$ weight system. The last two graphs are extended permutation graphs rather than just permutation graphs. Their expression in terms of permutation graphs is determined by the global structure of the original graph and depends on whether the vertices~$r$ and~$r+1$ belong to the same cycle or to two different ones. 
\end{definition}

It follows immediately from the definition that the universal $\fso$
weight system remains unchanged under the cyclic shift of a
permutation.

Below, we will write Eq.~(\ref{eq:rec-so1}) in the form
$$
w_\fso(\alpha)=w_\fso(\alpha')+w_\fso(\beta_1)-w_\fso(\beta_2)+w_\fso(\gamma_1)-w_\fso(\gamma_2).
$$
Here, permutations $\alpha$ and~$\alpha'$ are permutations of~$m$ elements, 
$\beta_1,\beta_2$ are permutations of~$m-1$ elements, and $ \gamma_1,\gamma_2$
are the extended permutation graphs on~$m-1$  elements.

\subsection{Lie algebra $\fso(N)$ with standard representation}

A given representation of a Lie algebra determines a weight system
with values in complex numbers. This weight system is obtained by
applying the representation composed of the $\text{Tr}$ operator.

A pair of permutations $\sigma,\alpha\in\S_m$, $\sigma$ being the standard cycle
$\sigma=(1,2,\dots,m)$, can be treated as an orientable hypermap
with a single (hyper)vertex, see details in~\cite{LZ03} or~\cite{KKL24}.
The hyperedges of the hypermap are the disjoint cycles of~$\alpha$.
This interpretation makes it natural to define the \emph{number of faces} $f(\alpha)$
of a permutation~$\alpha$ as the number of disjoint cycles in the product
$\alpha\sigma\in\S_m$. This number is equal to the number of connected components
of the boundary of the two-dimensional surface associated with the hypermap.

It is well known (see~\cite{BN95}) that in the standard representation
of the Lie algebra~$\fgl(N)$, the weight system~$w_{\fgl(N)}$ takes
a chord diagram~$D$ to the monomial~$N^{f(D)-1}$, $w^{{\rm St}}_{\fgl(N)}(D)=N^{f(D)-1}$.
This assertion immediately extends to the universal
$\fgl$ weight system and
arbitrary permutations (see~\cite{KL23}): for the substitution $C_i=N^{i-1}$, $i=1,2,\dots$,
the weight system~$w^{{\rm St}}_\fgl$ takes an arbitrary permutation~$\alpha$ to $N^{f(\alpha)-1}$,
$f(\alpha)$ being the number of faces of~$\alpha$.

For the $\fso(N)$ weight system in the standard representation, the assertion
acquires the following form. Recall first the result of D.~Bar-Natan~\cite{BN95}.
A \emph{state}~$s$ of a chord diagram~$D$ is a mapping $s:V(D)\to\{1,-1\}$
of the set of chords of~$D$ into a two-element set. Each state determines
a two-dimensional surface with boundary; the surface in question
is obtained by attaching
a disc to the circle of~$D$ and replacing each chord with an ordinary ribbon
attached to this disc,
if the state of the chord is~$1$, and by a half-twisted ribbon, if the state is~$-1$.
Note that if at least one chord is replaced by a half-twisted ribbon,
then the resulting surface will be nonorientable.
Denote by $f(D_s)$ the number of connected components of the boundary of the resulting surface,
and by~${\rm sign}(s)$ the \emph{sign} of~$s$, which is~$1$ if the number of chords in 
the state~$-1$ is even and $-1$ if this number is odd.

\begin{theorem}\cite{BN95}
For a chord diagram~$D$ of order~$n$,
$$w^{St}_{\fso(N)}(D) = \sum_{s:\{1,2,\dots,n\}\to\{1,-1\}} {\rm sign}(s)
N^{f(D_s)-1}\ ,
$$
where the sum is taken over all $2^n$ states for $D$.
\end{theorem}

This theorem means that the result is valid for arc diagrams, which are 
involutions without fixed points.
To extend this result to arbitrary permutations, define a \emph{state} of 
a permutation $\alpha\in \S_m$
as a mapping $s:\{1,2,\dots,m\}\to \{1,-1\}$. Note that according to this definition,
a chord diagram of order~$n$ has $2^{2n}$ states as opposed to the $2^n$
states in the original Bar-Natan definition.

Similarly to the chord diagram case, 
each state $s$ determines a hypermap. The boundary of the corresponding two-dimensional
surface with boundary is determined in terms of the graph of~$\alpha$
according to the following resolutions of each of the permuted elements:

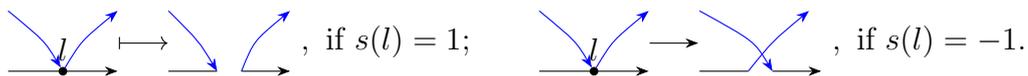
\[
\begin{tikzpicture}[scale=0.8]
\draw (0.9,0) node[v] (V) {};
\draw (0.9,0) node[above] {$l$};
\draw [->] (0,0) -- (1.8,0);
\draw[<-,blue] (1.8,1) .. controls +(-0.6,-0.5) .. (V);
\draw[->,blue] (0,1) .. controls (0.6,0.5) .. (V);
\end{tikzpicture}
\begin{tikzpicture}[scale=0.8]
\draw [|->] (0,0) -- (.8,0);
\end{tikzpicture}
\begin{tikzpicture}[scale=0.8]
\draw [-] (.0,0) -- (.8,0);
\draw [->] (1.2,0) -- (2,0);
\draw[<-,blue] (2,1) .. controls +(-0.6,-0.5) .. (1.2,0);
\draw[<-,blue] (.8,0) .. controls +(-0.3,0.5) .. (0,1);
\end{tikzpicture}\ ,\mbox{\ if\ }s(l)=1;\qquad
\begin{tikzpicture}[scale=0.8]
\draw (0.9,0) node[v] (V) {};
\draw (0.9,0) node[above] {$l$};
\draw [->] (0,0) -- (1.8,0);
\draw[<-,blue] (1.8,1) .. controls +(-0.6,-0.5) .. (V);
\draw[->,blue] (0,1) .. controls (0.6,0.5) .. (V);
\end{tikzpicture}
\begin{tikzpicture}[scale=0.8]
\draw [->] (0,0) -- (.8,0);
\end{tikzpicture}
\begin{tikzpicture}[scale=0.8]
\draw [-] (0,0) -- (.8,0);
\draw [->] (1.2,0) -- (2,0);
\draw[<-,blue] (1.8,1) .. controls +(-0.6,-0.5) .. (.8,0);
\draw[<-,blue] (1.2,0) .. controls +(-0.3,0.5) .. (0,1);
\end{tikzpicture}\ ,\mbox{\ if\ }s(l)=-1.\qquad
\]
For a point in the state~$1$ (respectively, $-1$), the end of the entering arrow 
of the permutation graph 
shifts to the left (respectively, to the right), and the beginning 
of the leaving arrow shifts to the right (respectively, to the left)
along the horizontal line.

In other words, to a disjoint cycle of length~$\ell$ an $\ell$-hyperedge
is associated. The latter is a $2\ell$-gon, each second edge of which is attached
to the disc, either in a non-twisted, or in a half-twisted fashion, depending
on the state of the corresponding permuted element.
Denote by $f(\alpha_s)$ the number of connected components of the obtained curve
(that is, the number of connected components of the boundary of the resulting surface).
The \emph{sign} of~$s$, ${\rm sign}(s)$, is~$1$ if the number of legs in 
the state~$-1$ is even and $-1$ if this number is odd.

\begin{theorem}
For a permutation~$\alpha\in\S_m$, we have
$$w^{St}_{\fso(N)}(\alpha) = \sum_{s:\{1,2,\dots,m\}\to\{1,-1\}}{\rm sign}(s)
N^{f(\alpha_s)-1}\ ,
$$
where the sum is taken over all the $2^m$ states for $\alpha$.
\end{theorem}

\begin{proof} The weight system corresponding to
the standard representation of the Lie algebra $\fso(N)$
is the composition of the mapping $\frac1{N}{\rm Tr}$ with 
the standard representation of $U\fso(N)$.
The generator $X_{ij}$, $1\le i,j\le N$ is represented by the difference of
the matrix units $X_{ij}=E_{ij}-E_{ji}$. A product of several matrix units $E_{ij}$
is either a zero matrix, or a matrix unit, and its trace is nonzero
iff it is a diagonal matrix unit. In the latter case, the trace is~$1$.

For a given state of a permutation of~$m$ elements, 
associate an index between~$1$ and~$N$ to each 
connected component of the boundary of the corresponding hypermap,
obtained as described above.
To such an assignment, a unique product of~$m$ matrix units $E_{ij}$ is associated:
the indices $i$ and $j$ at each leg are determined by the two boundary
components passing through this leg (which may well coincide).
Here, from the difference $X_{ij}=E_{ij}-E_{ji}$ we take $E_{ij}$ if the
corresponding leg is in state~$1$, and $-E_{ji}$ otherwise.
In this way we establish a one-to-one correspondence between states with
numbered boundary components and those monomials in the matrix units 
in the product~(\ref{eq:wsonsigma}) which yield a nonzero contribution
to the trace of the resulting sum of the matrices in the standard representation.

Now, since the index of each boundary component varies from~$1$ to~$N$,
the contribution to the trace of a given state~$s$ of a permutation~$\alpha$
is~$N^{f(\alpha_s)}$. Dividing by~$N$ and taking into account that the sign
of the contribution is plus or minus depending on the parity
of the number of negative states
of the legs, we obtain the desired.
\end{proof}

In particular, for standard cycles we immediately obtain by induction over
the number of permuted elements

\begin{corollary}  The standard representation of~$\fso$ on the standard cycles
acquires the form
$$w_{\fso(N)}^{St}((1,2,\dots,m))=
w_{\fso(N)}^{St}\left(
  \begin{tikzpicture}
\draw (0,0) node[v] (V1) {}  node[above] {$\scriptsize 1$} (0.6,0) node[v] (V2) {}  node[above] {$\scriptsize 2$} (1.2,0) node[v] (V3) {} (2.4,0) node[v] (V4) {} (3,0) node[v] (V5) {}  node[above] {$\scriptsize m$}
(1.8,0) node {$\cdots$};
\draw[->] (-0.4,0) -- (3.6,0);
 \draw[->,blue] (V1) .. controls +(.3,.2) and +(-.3,.2) ..  (V2);
 \draw[->,blue] (V2) .. controls +(.3,.2) and +(-.3,.2) .. (V3);
 \draw[->,blue] (V4) .. controls +(.3,.2) and +(-.3,.2) ..  (V5);
 \draw[->,blue] (V5) .. controls +(-.5,-.5) and +(.5,-.5) .. (V1);
 \end{tikzpicture}
 \right)=\begin{cases}
 \frac{(N-1)^m+1-N}{N}, &\qquad m\text{ is odd};\\
 \frac{(N-1)^m-1+N^2}{N}, &\qquad m\text{ is even}.
 \end{cases}
 $$
 In other words, the value of the~$\fso(N)$ weight system in the standard representation
 on a given permutation~$\alpha$ can be obtained by substituting these
 values of Casimirs to $\fso(N)$.
\end{corollary}

Below, for~$k$ even, we will use notation
 \begin{equation}\label{e-P}
       P_k(N)=\frac{(N-1)^k-1+N^2}{N}
  \end{equation} 
for polynomials which are the values of the~$\fso(N)$ weight system on standard
cycles of even length.

\section{Recurrence rule for pre-chromatic substitution}


\subsection{Hopf algebras of permutations}

In what follows, we will distinguish certain classes of permutations.

\begin{definition}\rm
    An \emph{ascent} of a permutation~$\alpha\in\S_m$
    is a number $i$, $0<i<m$ such that $\alpha(i)>i$. We denote
    the number of ascents in~$\alpha$ by $a(\alpha)$.

A cyclic permutation~$\alpha\in\S_m$ is \emph{positive} (respectively,
\emph{negative}) if $a(\alpha)=m-1$ (respectively, if $a(\alpha)=1$).
In other words, the permutation $(1,2,\dots,m)$ is the only positive cyclic permutation
of~$m$ elements, while $(1,m,m-1,\dots,2)$ is the only negative cyclic permutation
of~$m$ elements.
Note that the cyclic permutation $(12)\in\S_2$ is both positive and negative.
By convention, we also consider the cyclic permutation $(1)\in\S_1$ 
as being both positive and negative. 
    
    A  permutation~$\alpha\in\S_m$ 
    is \emph{positive} (respecitvely, \emph{negative}) if  $a(\alpha)=m-c(\alpha)$
    (respectively, if $a(\alpha)=c(\alpha)$), where $c(\alpha)$ is the number of cycles in $\alpha$. This is equivalent to saying that each disjoint cycle in~$\alpha$
    is positive (respectively, negative).
The restriction
$\alpha|_U$ of a positive (respectively, negative) permutation~$\alpha\in\S_m$ to an arbitrary subset 
$U\subset\{1,2,\dots,m\}$ is a positive (respectively, negative) permutation.

    A permutation is \emph{monotone} if each of its disjoint
    cycles is either positive or negative. The restriction of a monotone permutation
    to any subset of permuted elements also is a monotone permutation.
\end{definition}

In~\cite{KKL24}, several Hopf algebras of cyclic equivalence classes of permutations
were introduced. The
cyclic equivalence of permutations of~$m$ elements is defined as follows:
\begin{itemize}
    \item any cyclic shift of a permutation is equivalent to it, 
    $\alpha\sim\sigma^{-1}\alpha\sigma$;
    \item the concatenation product of any two permutations is equivalent
    to the concatenation product of any pair of permutations 
    cyclically equivalent to the factors.
\end{itemize}

Cyclic equivalence classes of permutations span the rotational Hopf algebra~$\cA$
of permutations.
Denote the equivalence class of a permutation~$\alpha$
by $[\alpha]$. Let $$
\cA=\cA_0\oplus \cA_1\oplus \cA_2\oplus\dots
$$
be the infinite dimensional vector space, which is the direct sum of
the finite dimensional vector spaces $\cA_m$, $m=0,1,2,\dots$,
the vector space~$\cA_m$ being freely spanned by the equivalence
classes of permutations of~$m$ elements. 
The concatenation product of permutations makes $\cA$ into
an infinite dimensional commutative graded algebra.
The comultiplication $\mu:\cA\to\cA\otimes\cA$ is defined as follows:
for a generator $[\alpha]\in \cA_m$ take a permutation $\alpha\in\S_m$
representing this generator and set
$$
\mu([\alpha])=[\mu(\alpha)]=\sum_{I\sqcup J=V(\alpha)}[{\alpha|_{I}]\otimes[\alpha|_{J}}],
$$
where the summation is carried over all the ways to represent the set
$V(\alpha)$ of disjoint cycles of~$\alpha$ into a disjoint
union of two subsets $I,J$, and $\alpha|_{I}$
denotes the restriction of $\alpha$ to the subset 
of elements contained in the cycles from $I$.
Here, following~\cite{KoL23}, we define the \emph{restriction} $\alpha|_U\in\S_{|U|}$
of a permutation~$\alpha\in\S_m$
to a subset $U\subset\{1,2,\dots,m\}$ as a permutation of the elements
of~$U$ preserving their relative cyclic order inside each disjoint 
cycle in~$\alpha$, composed with the subsequent renumbering to $\{1,2,\dots,|U|\}$.

The rotational Hopf algebra~$\cA^+$
of positive permutations, and the rotational Hopf algebra~$\cB$
of chord diagrams, so that $\cB\subset\cA^+\subset\cA$, 
are Hopf subalgebras of~$\cA$,
see details in~\cite{KKL24}.
We add to this list the rotational Hopf algebra $\cA^-$ of negative permutations
and the rotational Hopf algebra $\cA^\pm$ of monotone permutations, so that
$\cA^-\subset\cA^\pm\subset\cA$ and $\cA^+\subset\cA^\pm$. Each of these
Hopf algebras is spanned by the cyclic equivalence classes of permutations
of corresponding types.

\subsection{Prechromatic substitution}

In~\cite{KKL24}, the prechromatic substitution for the Casimir elements~$C_1,C_2,\dots$
in the universal $\fgl$ weight system
acquires the form $C_k=p_kN^{k-1}$, $k=1,2,\dots$. This substitution can be considered as a
perturbation of the $\fgl(N)$ weight system in the standard representation,
the latter being given by $C_k=N^{k-1}$, $k=1,2,\dots$.

For the universal $\fso$ weight system, the corresponding substitution would require 
multiplication of the standard representation value for the Casimirs~$S_k$ by~$p_k$.

\begin{definition}[prechromatic substitution]
        For a permutation $\alpha\in \mathbb{S}_m$, $c(\alpha)$ being the number of cycles in 
        $\alpha$, we define the \emph{prechromatic substitution} by
    \[
    Y(\alpha)=N^{c(\alpha)-m}w_\fso(\alpha)|_{S_0=N,S_k=p_k\cdot P_k(N), k=2,4,6\dots},
    \]
with~$P_k(N)$ given by Eq.~(\ref{e-P}).
\end{definition} 

Obviously, $Y$ can be written as a finite sum
    \[
        Y(\alpha)=Y_0(\alpha)+Y_1(\alpha)N^{-1}+Y_2(\alpha)N^{-2}+\cdots
    \]
    where each $Y_k(\alpha)$ is a polynomial in $p_2,p_4,p_6\dots$.

In this paper, our main goal is to describe the properties of 
the leading term~$Y_0$  of the prechromatic substitution.
We are planning to analyze the next terms,
which we consider to be of importance, in future papers.


It is easy to see that 

\begin{lemma} The constant terms in~$N$ (that is, coefficients of~$N^0$) for the 
prechromatic substitution are determined 
by the leading term $p_kN^{k-1}$, $k=2,4,6,\dots$, of the polynomials~$P_k$
given by~Eq.(\ref{e-P}), for each~$\alpha$.
\end{lemma}

Specializing the defining relations for the $\fso$ weight system to 
the prechromatic substitution we immediately obtain

\begin{theorem}[Recurrence rule for pre-chromatic substitution]
The recursion for $w_{\fso}$, when rewritten in terms of $Y$, acquires the following
form:
    \begin{itemize}
    \item the invariant $Y$ takes equal values on all permutations in one equivalence
class;
    \item  it is multiplicative with respect to concatenation product of permutations;
    \item the value of $Y$ on the standard $m$-cycle of even length
    $m=2n$, is equal to $p_{2n}P_{2n}(N)$;
    \item if $r$ and $r + 1$ belong to different cycles, then we have
\begin{multline}\label{eq:rec-X1}
    Y\left(~
\begin{tikzpicture}[scale=0.8]
\draw (0.6,0) node[v] (V1) {} (1.2,0) node[v] (V2) {};
\draw  (0.5,-.25) node {\scriptsize$r$}  
 (1.3,-.25) node {\scriptsize$r{+}1$};
\draw [->] (0,0) -- (1.8,0);
\draw[->,blue] (0,0.7) .. controls (0.6,0.5) .. (V2);
\draw[-,blue] (V2) .. controls (0.6,0.9) .. (0,1.4);
\draw[->,blue] (1.8,1.4) .. controls (1.2,0.9) .. (V1);
\draw[-,blue] (V1) .. controls (1.2,0.5) .. (1.8,0.7);
\end{tikzpicture}
 ~\right)=Y\left(~
\begin{tikzpicture}[scale=0.8]
\draw  (0.7,0) node[v] (V1) {} (1.1,0) node[v] (V2) {};
\draw [->] (0,0) -- (1.8,0);
\draw[->,blue] (0,0.7) .. controls (0.4,0.4) .. (V1);
\draw[-,blue] (V1) .. controls (0.4,0.9) .. (0,1.4);
\draw[->,blue] (1.8,1.4) .. controls (1.4,0.9) .. (V2);
\draw[-,blue] (V2) .. controls (1.4,0.4) .. (1.8,0.7);
\end{tikzpicture}
 ~\right)+Y\left(~
\begin{tikzpicture}[scale=0.8]
\draw (0.9,0) node[v] (V) {};
\draw [->] (0,0) -- (1.8,0);
\draw[-,blue] (0,0.7) .. controls +(0.5,-0.5) and +(-0.5,-0.5) .. (1.8,0.7);
\draw[->,blue] (1.8,1.4) .. controls +(-0.6,-0.7) .. (V);
\draw[-,blue] (V) .. controls +(-0.3,0.7) .. (0,1.4);
\end{tikzpicture}
 ~\right)-Y\left(~
\begin{tikzpicture}[scale=0.8]
\draw (0.9,0) node[v] (V) {};
\draw [->] (0,0) -- (1.8,0);
\draw[->,blue] (0,0.7) .. controls (0.4,0.5) .. (V);
\draw[-,blue] (V) .. controls (1.4,0.5) .. (1.8,0.7);
\draw[-,blue] (1.8,1.4) .. controls (0.9,0.2) .. (0,1.4);
\end{tikzpicture}
 ~\right)
 \\
+Y\left(~
\begin{tikzpicture}[scale=0.8]
\draw (0.9,0) node[v] (V) {};
\draw [->] (0,0) -- (1.8,0);
\draw[blue] (0,0.7) .. controls (0.6,0.5) and +(-0.5,-0.7) .. (1.8,1.4);
\draw[<-,blue] (V) .. controls (0.6,0.9) .. (0,1.4);
\draw[-,blue] (V) .. controls (1.2,0.5) .. (1.8,0.7);
\end{tikzpicture}
 ~\right)-Y\left(~
\begin{tikzpicture}[scale=0.8]
\draw (0.9,0) node[v] (V) {};
\draw [->] (0,0) -- (1.8,0);
\draw[-,blue] (1.8,0.7) .. controls +(-0.6,-0.4) and (0.6,0.9) .. (0,1.4);
\draw[->,blue] (1.8,1.4) .. controls (1.2,0.7) .. (V);
\draw[blue] (V) .. controls (0.6,0.5) .. (0,0.7);
\end{tikzpicture}
 ~\right)\end{multline}

     \item if $r$ and $r + 1$ belong to the same cycle and $r+1\ne \alpha^{\pm1}(r)$, we have
\begin{multline}\label{eq:rec-X2}
    Y\left(~
\begin{tikzpicture}[scale=0.8]
\draw (0.6,0) node[v] (V1) {} (1.2,0) node[v] (V2) {};
\draw  (0.5,-.25) node {\scriptsize$r$}  
 (1.3,-.25) node {\scriptsize$r{+}1$};
\draw [->] (0,0) -- (1.8,0);
\draw[->,blue] (0,0.7) .. controls (0.6,0.5) .. (V2);
\draw[-,blue] (V2) .. controls (0.6,0.9) .. (0,1.4);
\draw[->,blue] (1.8,1.4) .. controls (1.2,0.9) .. (V1);
\draw[-,blue] (V1) .. controls (1.2,0.5) .. (1.8,0.7);
\end{tikzpicture}
 ~\right)=Y\left(~
\begin{tikzpicture}[scale=0.8]
\draw  (0.7,0) node[v] (V1) {} (1.1,0) node[v] (V2) {};
\draw [->] (0,0) -- (1.8,0);
\draw[->,blue] (0,0.7) .. controls (0.4,0.4) .. (V1);
\draw[-,blue] (V1) .. controls (0.4,0.9) .. (0,1.4);
\draw[->,blue] (1.8,1.4) .. controls (1.4,0.9) .. (V2);
\draw[-,blue] (V2) .. controls (1.4,0.4) .. (1.8,0.7);
\end{tikzpicture}
 ~\right)+N^{-2}\left(~ Y\left(~
\begin{tikzpicture}[scale=0.8]
\draw (0.9,0) node[v] (V) {};
\draw [->] (0,0) -- (1.8,0);
\draw[-,blue] (0,0.7) .. controls +(0.5,-0.5) and +(-0.5,-0.5) .. (1.8,0.7);
\draw[->,blue] (1.8,1.4) .. controls +(-0.6,-0.7) .. (V);
\draw[-,blue] (V) .. controls +(-0.3,0.7) .. (0,1.4);
\end{tikzpicture}
 ~\right)-Y\left(~
\begin{tikzpicture}[scale=0.8]
\draw (0.9,0) node[v] (V) {};
\draw [->] (0,0) -- (1.8,0);
\draw[->,blue] (0,0.7) .. controls (0.4,0.5) .. (V);
\draw[-,blue] (V) .. controls (1.4,0.5) .. (1.8,0.7);
\draw[-,blue] (1.8,1.4) .. controls (0.9,0.2) .. (0,1.4);
\end{tikzpicture}
 ~\right)~\right)
 \\
+N^{-1}\left(Y\left(~
\begin{tikzpicture}[scale=0.8]
\draw (0.9,0) node[v] (V) {};
\draw [->] (0,0) -- (1.8,0);
\draw[blue] (0,0.7) .. controls (0.6,0.5) and +(-0.5,-0.7) .. (1.8,1.4);
\draw[<-,blue] (V) .. controls (0.6,0.9) .. (0,1.4);
\draw[-,blue] (V) .. controls (1.2,0.5) .. (1.8,0.7);
\end{tikzpicture}
 ~\right)-Y\left(~
\begin{tikzpicture}[scale=0.8]
\draw (0.9,0) node[v] (V) {};
\draw [->] (0,0) -- (1.8,0);
\draw[-,blue] (1.8,0.7) .. controls +(-0.6,-0.4) and (0.6,0.9) .. (0,1.4);
\draw[->,blue] (1.8,1.4) .. controls (1.2,0.7) .. (V);
\draw[blue] (V) .. controls (0.6,0.5) .. (0,0.7);
\end{tikzpicture}
 ~\right) ~\right)\end{multline}

 \item Finally, in the exceptional case when $\alpha(r+1)=r$ we have
\begin{equation}\label{eq:rec-X3}
Y\left(~
\begin{tikzpicture}[scale=0.8]
\draw (0.6,0) node[v] (V1) {} (1.2,0) node[v] (V2) {};
\draw  (0.6,-.25) node {\scriptsize$r$}  
 (1.2,-.25) node {\scriptsize$r{+}1$};
\draw [->] (0,0) -- (1.8,0);
\draw[->,blue] (0,1) .. controls (0.7,0.6) .. (V2);
\draw[-,blue] (V1) .. controls (1.1,0.6) .. (1.8,1);
\draw[->,blue] (V2) .. controls (0.9,0.2) .. (V1);
\end{tikzpicture}
 ~\right)=
 Y\left(~
\begin{tikzpicture}[scale=0.8]
\draw  (0.6,0) node[v] (V1) {} (1.2,0) node[v] (V2) {};
\draw [->] (0,0) -- (1.8,0);
\draw[->,blue] (0,1) .. controls (0.4,0.6) .. (V1);
\draw[-,blue] (V2) .. controls (1.4,0.6) .. (1.8,1);
\draw[->,blue] (V1) .. controls (0.9,0.3) .. (V2);
\end{tikzpicture}
 ~\right)+
(2N^{-1}-1)\,Y\left(~
\begin{tikzpicture}[scale=0.8]
\draw (0.9,0) node[v] (V) {};
\draw [->] (0,0) -- (1.8,0);
\draw[-,blue] (V) .. controls (1.2,0.5) .. (1.8,1);
\draw[->,blue] (0,1) .. controls (0.6,0.5) .. (V);
\end{tikzpicture}
 ~\right).
\end{equation}
    \end{itemize}
\end{theorem}

Our first main result is

\begin{theorem}\label{thm:Xhopf}
    The leading term~$Y_0$ of the  prechromatic substitution determines
    filtered Hopf algebra homomorphisms from the Hopf algebras~$\cA^+,\cA^-,\cA^\pm$
    of the rotational Hopf algebras of, respectively, positive, negative, and monotone
    permutations to the Hopf algebra $\mathbb{C}[p_2,p_4,\dots]$ of polynomials
    in infinitely many variables.
\end{theorem}

\begin{proof}

Below, we will use a linear order on the set of 
equivalence classes of permutations in~$\S_m$, that is,
on the set of basic vectors
in the spaces $\cA_m$, defined in~\cite{KKL24}. 
Then we argue by a double induction
with respect to~$m$ and this linear ordering.

Associate to the equivalence class $[\alpha]$ 
of a permutation~$\alpha$ its
canonical form in the following
way. Write each permutation in the equivalence class
as a product of disjoint cycles, where for each cycle 
we chose the lexicographically minimal presentation (the one with the smallest first element),
and write the cycles in the lexicographic order (that is, the order in which their
first elements increase). Now, among all the permutations 
in the equivalence class, choose the permutation with the lexicographically smallest canonical
form. This will be the \emph{canonical form}~$\bar\alpha$ of~$[\alpha]$. The lexicographic ordering of canonical
forms defines a linear ordering on the basic vectors in~$A_m$.

The multiplicativity $Y_0(\alpha_1\sharp \alpha_2)=Y_0(\alpha_1)Y_0(\alpha_2)$ is trivial, since
the weight system $w_{\fso}$ is multiplicative.

The comultiplicativity means that $(Y_0\otimes Y_0)\circ\mu=Y_0|_{p_k=1\otimes p_k+p_k\otimes1}$. It can be proved as follows. Denote the mapping on the left by $\theta$,
so that, for a given permutation $\alpha$, we have
\[
\theta:\alpha\mapsto \sum_{I\sqcup J=V(\alpha)}
Y_0(\alpha|_I)\otimes Y_0(\alpha|_J),
\]
where the summation is carried over all ordered partitions of the
set~$V(\alpha)$ of disjoint cycles of~$\alpha$ into two
disjoint subsets. 

\begin{lemma}\label{lem:X}
    We have the following recurrence for $\theta$:
    \begin{itemize}
        \item under (\ref{eq:rec-X1}), we have
        \[
        \theta(\alpha)=\theta(\alpha')+\theta(\beta_1)-\theta(\beta_2)+\theta(\gamma_1)-\theta(\gamma_2)
        \]
        \item under (\ref{eq:rec-X2}), we have
        \[
        \theta(\alpha)=\theta(\alpha')
        \]
        \item under (\ref{eq:rec-X3}), we have
        \[
        \theta(\alpha)=\theta(\alpha')-\theta(\gamma_2)
        \]
    \end{itemize}
\end{lemma}
For the standard cyclic permutation of length~$m$, $Y_0$ equals $p_m$, which forms the induction base. For any other permutation, one can choose two neighboring legs 
such that the permutations on the right side are lexicographically smaller.

\noindent
\emph{Proof} of Lemma~\ref{lem:X}.

Consider the possibilities case by case. In all the cases we assume that the
recurrence is applied to the canonical form~$\bar\alpha$ of the permutation~$\alpha$.
\begin{itemize}
    \item If the pair $r,r+1$ is like in recursion~(\ref{eq:rec-X1}), so that $r$  and $r+1$ belong to two different
    disjoint cycles of~$\alpha$, then let $v_1\in V(\alpha)$
    be the cycle containing~$r$, and $v_2\in V(\alpha)$ 
    be the cycle containing~$r+1$.  The disjoint cycles~$V(\alpha')$
    naturally correspond one-to-one to the disjoint cycles
    $V(\alpha)$, while each of the permutations $\beta_1,\beta_2,\gamma_1,\gamma_2$
    has one disjoint cycle less: the pair of cycles $v_1,v_2$ is replaced by a single cycle. Now, we have
\begin{eqnarray*}
   && \theta(\alpha)-\theta(\alpha')-\theta(\beta_1)+\theta(\beta_2)-\theta(\gamma_1)+\theta(\gamma_2)=\\
   && \sum_{{I\sqcup J=V(\alpha)\atop \{v_1,v_2\}\subset I}}
    Y_0(\alpha|_I)\otimes(Y_0(\alpha|_J)-Y_0(\alpha'|_J)
    -Y_0(\beta_1|_J)+Y_0(\beta_2|_J)-Y_0(\gamma_1|_J)+Y_0(\gamma_2|_J))\\
   && +\sum_{{I\sqcup J=V(\alpha)\atop \{v_1,v_2\}\subset J}}
   (Y_0(\alpha|_I)-Y_0(\alpha'|_I)
    -Y_0(\beta_1|_I)+Y_0(\beta_2|_I)-Y_0(\gamma_1|_I)+Y_0(\gamma_2|_I)) \otimes Y_0(\alpha|_J)\\
    &&+\sum_{{I\sqcup J=V(\alpha)\atop v_1\in I,v_2\in J
    \text{ or }v_2\in I,v_1\in J}}(Y_0(\alpha|_I)\otimes
    Y_0(\alpha|_J)-Y_0(\alpha'|_I)
    \otimes Y_0(\alpha'|_J)).
\end{eqnarray*}
Each of the three summands on the right is zero: for the first two
sums the expression in brackets is~$0$ due to recurrence~(\ref{eq:rec-so1}) applied to the permutations $\alpha|_J$ and $\alpha|_I$,
respectively,
while the third sum vanishes because $\alpha|_I=\alpha'|_I,
\alpha|_J=\alpha'|_J$.
 \item If the pair $r,r+1$ is like in recurrence~(\ref{eq:rec-X2}), so that $r$ and $r+1$ belong to the same
    disjoint cycle of~$\alpha$, then let $v\in V(\alpha)$
    be the cycle containing this pair of elements.  
    The disjoint cycles~$V(\alpha')$
    naturally correspond one-to-one to the disjoint cycles
    $V(\alpha)$. We claim that for an arbitrary subset $I\subset V(\alpha)$ the values $Y_0(\alpha|_I)$ and $Y_0(\alpha'|_I)$
    coincide. Indeed, if $I$ does not contain~$v$, then 
    $\alpha|_I=\alpha'|_I$. If~$I$ contains~$v$, then the 
    two restrictions coincide no longer. However, in this case
    the permutations $\alpha|_I$ and $\alpha'|_I$ are related by 
    the same recurrence, and the equality
    follows from the induction hypothesis.
     \item If the pair $r,r+1$ is like in recursion~(\ref{eq:rec-X3}), so that $r$ and $r+1$ are two successive elements
     in the same disjoint cycle, then let $v\in V(\alpha)$
    be this cycle.  The disjoint cycles~$V(\alpha')$ and $V(\alpha_2)$
    naturally correspond one-to-one to the disjoint cycles
    $V(\alpha)$, while the permutation $\alpha_1$
    has one disjoint cycle more: the cycle $v$ is replaced by two cycles. Thus, for any subset $I\subset V(\alpha)$ not containing~$v$, we have $\alpha|_I=\alpha'|_I=\beta_2|_I$.
    As a consequence,
\begin{eqnarray*}
   && \theta(\alpha)-\theta(\alpha')+\theta(\beta_2)=\\
   && \sum_{{I\sqcup J=V(\alpha)\atop v\in I}}
   (Y_0(\alpha|_I)-Y_0(\alpha'|_I)
   +Y_0(\beta_2|_I))\otimes  Y_0(\alpha|_J)\\
   && \sum_{{I\sqcup J=V(\alpha)\atop v\in J}}
    Y_0(\alpha|_J)\otimes(Y_0(\alpha|_J)-Y_0(\alpha'|_J)
   +Y_0(\beta_2|_J)).
\end{eqnarray*}
Each of the two summands on the right vanishes: the expression in brackets is~$0$ due to recurrence~(\ref{eq:rec-X3}) applied to the permutations $\alpha|_J$ and $\alpha|_I$,
respectively.
\end{itemize}
Lemma~\ref{lem:X} is proved, which completes the proof of Theorem~\ref{thm:Xhopf}.
\end{proof}

\subsection{Values of $Y_0$ on cyclic permutations}
In this section we compute explicitly the value of~$Y_0$
on cyclic permutations. In~\cite{KKL24}, the following formula is obtained for the
leading term $X_0$ of the universal $\fgl$ 
weight system on a cyclic permutation $\alpha\in\S_m$ having $m-k$ ascents, $a(\alpha)=m-k$,
under the substitution $C_k=p_kN^{k-1}$, $k=1,2,\dots$:

\begin{equation}\label{e-X0}
X_0(\alpha)=p_{m}-{k-1\choose1}p_{m-1}+{k-1\choose2}p_{m-2}-\dots+(-1)^{k-1}{k-1\choose k-1}p_{m-k+1}.
\end{equation}

In order to describe the values of~$Y_0$ on cyclic permutations,
let us introduce the following linear operator $A$ on the
space $\mathbb{C}[t]$ of polynomials in the single variable~$t$
as the result of iterating the operator
$$
t^n\mapsto\left\{\begin{array}{cl}
0&n=1\\
t^n&n\text{ is even}\\
\frac{t^2}2(t^{n-2}-(t-1)^{n-2})&\text{otherwise}
\end{array}\right\}
$$
until obtaining an even polynomial,
so that
$$
A:t\mapsto 0;\qquad A:t^3\mapsto \frac12t^2;\qquad A:t^5\mapsto 
\frac32t^4-\frac14t^2;\qquad\dots
$$
The operator~$A$ is a projection of the space $\mathbb{C}[t]$ to
the subspace of even polynomials.

\begin{proposition}
    If $\alpha\in S_m$ is a cyclic permutation, that is, $c(\alpha)=1$,
    then the value $Y$ of the prechromatic invariant on it
    has the form 
    $$
    Y(\alpha)=Y_0(p_2,p_4\dots)+
    Y_1(p_2,p_4\dots)N^{-1}+\dots,
    $$
where $Y_0$ is a linear polynomial in $p_2,p_4\dots$ obtained by making the following
substitution for the variables~$p_{2k-1}$ with odd indices to the value~$X_0(\alpha)$
given by Eq.~(\ref{e-X0}) above:
\begin{eqnarray*}
        &&p_1:=0\\
        &&p_3:=\frac{1}{2}p_2\\
        &&p_5:=\frac{3}{2}p_4-\frac{1}{4}p_2\\
        &&\dots\\
        &&p_{2n+1}:=A(t^{2n+1})|_{t^{2k}:=p_{2k},k=1,2,\dots}
    \end{eqnarray*}

\begin{remark}
     Under substitution $p_2=2x,p_m=x$, for $m=4,6\dots$, permutation $\alpha\in\S_m$ having $m-k$ ascents, we have \begin{itemize}
         \item $p_{2n+1}=x$ ;
         \item for $1<k<m$, $Y_0(\alpha)$ vanish;
         \item for positive permutation $\alpha$, which means $k=1$, we have $Y_0=x$;
         \item for negative permutation $\alpha$, which means $k=m-1$, we have $Y_0=(-1)^{m}x$.
     \end{itemize}
\end{remark}
\end{proposition}
\begin{proof}
    Let us prove this by induction.

In the first part, we prove that $Y_0(\alpha)$ as the polynomial of $p_1,p_2,p_3\dots$ 
is given by the same formula Eq.~(\ref{e-X0}) as $X_0(\alpha)$.

Recall the recurrences (\ref{eq:rec-X2}) and (\ref{eq:rec-X3}).
\begin{itemize}
         \item if $r$ and $r + 1$ belong to the same cycle and $r+1\ne \alpha^{\pm1}(r)$, we have
\begin{equation}\label{eq:X0-samecyc1}
    Y_0\left(~
\begin{tikzpicture}[scale=0.8]
\draw (0.6,0) node[v] (V1) {} (1.2,0) node[v] (V2) {};
\draw  (0.5,-.25) node {\scriptsize$r$}  
 (1.3,-.25) node {\scriptsize$r{+}1$};
\draw [->] (0,0) -- (1.8,0);
\draw[->,blue] (0,0.7) .. controls (0.6,0.5) .. (V2);
\draw[-,blue] (V2) .. controls (0.6,0.9) .. (0,1.4);
\draw[->,blue] (1.8,1.4) .. controls (1.2,0.9) .. (V1);
\draw[-,blue] (V1) .. controls (1.2,0.5) .. (1.8,0.7);
\end{tikzpicture}
 ~\right)=Y_0\left(~
\begin{tikzpicture}[scale=0.8]
\draw  (0.7,0) node[v] (V1) {} (1.1,0) node[v] (V2) {};
\draw [->] (0,0) -- (1.8,0);
\draw[->,blue] (0,0.7) .. controls (0.4,0.4) .. (V1);
\draw[-,blue] (V1) .. controls (0.4,0.9) .. (0,1.4);
\draw[->,blue] (1.8,1.4) .. controls (1.4,0.9) .. (V2);
\draw[-,blue] (V2) .. controls (1.4,0.4) .. (1.8,0.7);
\end{tikzpicture}
 ~\right)\end{equation}

 \item in the exceptional case when $\alpha(r+1)=r$ we have
\begin{equation}\label{eq:X0-samecyc2}
Y_0\left(~
\begin{tikzpicture}[scale=0.8]
\draw (0.6,0) node[v] (V1) {} (1.2,0) node[v] (V2) {};
\draw  (0.6,-.25) node {\scriptsize$r$}  
 (1.2,-.25) node {\scriptsize$r{+}1$};
\draw [->] (0,0) -- (1.8,0);
\draw[->,blue] (0,1) .. controls (0.7,0.6) .. (V2);
\draw[-,blue] (V1) .. controls (1.1,0.6) .. (1.8,1);
\draw[->,blue] (V2) .. controls (0.9,0.2) .. (V1);
\end{tikzpicture}
 ~\right)=
 Y_0\left(~
\begin{tikzpicture}[scale=0.8]
\draw  (0.6,0) node[v] (V1) {} (1.2,0) node[v] (V2) {};
\draw [->] (0,0) -- (1.8,0);
\draw[->,blue] (0,1) .. controls (0.4,0.6) .. (V1);
\draw[-,blue] (V2) .. controls (1.4,0.6) .. (1.8,1);
\draw[->,blue] (V1) .. controls (0.9,0.3) .. (V2);
\end{tikzpicture}
 ~\right)-Y_0\left(~
\begin{tikzpicture}[scale=0.8]
\draw (0.9,0) node[v] (V) {};
\draw [->] (0,0) -- (1.8,0);
\draw[-,blue] (V) .. controls (1.2,0.5) .. (1.8,1);
\draw[->,blue] (0,1) .. controls (0.6,0.5) .. (V);
\end{tikzpicture}
 ~\right).
\end{equation}
\end{itemize}
Let $\alpha$ be the minimal cyclic permutation for which the assertion of the
Proposition is not proved yet.  Let $r$  be the minimal element such that $\alpha(r)> r + 1$. (If such an element $r$ does not exist, then $\alpha$ is a standard cycle, for which the assertion of the Proposition is obvious.) There are two possibilities:
\begin{itemize}
    \item if $\alpha(\alpha(r))\ne \alpha(r)-1$, then
    apply \eqref{eq:X0-samecyc2} to the legs 
    $\alpha(r)-1,\alpha(r)$. The value of~$Y_0$ does not change;
    \item if $\alpha(\alpha(r))=\alpha(r)-1$, then apply \eqref{eq:X0-samecyc1} 
    to the legs 
    $\alpha(r)-1,\alpha(r)$. The number of ascents in the permutation $\alpha'$, the 
    first permutation on the right, is one more than in $\alpha$, and the second 
    element $\alpha_2$ is a cyclic permutation of $m-1$ elements, having
    the same  number of 
    ascents. By induction:
    \begin{eqnarray*}
    Y_0(\alpha)=&(p_m-{m-a(\alpha)-2\choose1}p_{m-1}+{m-a(\alpha)-2\choose2}p_{m-2}-\dots)+\\
    &(p_{m-1}-{m-a(\alpha)-2\choose1}p_{m-2}+{m-a(\alpha)-2\choose2}p_{m-3}-\dots)\\
    =&(p_m-{m-a(\alpha)-1\choose1}p_{m-1}+{m-a(\alpha)-1\choose2}p_{m-2}-\dots)
    \end{eqnarray*}
\end{itemize}
By induction, we have
\[
    Y_0(\alpha)=p_{m}-{k-1\choose1}p_{m-1}+{k-1\choose2}p_{m-2}-\dots+(-1)^{k-1}{k-1\choose k-1}p_{m-k+1}.
\]

Next we express the variables with odd indices $p_{2k+1}$ as polynomials in $p_1,p_2,\dots,p_{2k}$ and prove the substitution formula.

We have $p_1=0$, since $S_1=0$.

For the standard cycle $\sigma_{2k+1}=(1,2,\dots,2k+1)\in\S_{2k+1}$ of length $2k+1$, we have 
\[Y_0(\sigma^{-1}_{2k+1})=(-1)^{2k+1}Y_0(\sigma_{2k+1})=-p_{2k+1}.\]
On the other hand, $\sigma^{-1}_{2k+1}$ has one ascent, hence
\[
Y_0(\sigma^{-1}_{2k+1})=p_{2k+1}-{2k-1\choose1}p_{2k}+{2k-1\choose2}p_{2k-1}-\dots+(-1)^{2k-1}{2k-1\choose 2k-1}p_{2}.
\]
Applying the substitution $p_n=t^n,n=1,2,\dots,2k$, we obtain
\begin{align*}
    p_{2k+1}=&\frac{1}{2}\left({2k-1\choose1}p_{2k}-{2k-1\choose2}p_{2k-1}-\dots+(-1)^{2k-2}{2k-1\choose 2k-1}p_{2}\right)\\
    =&\frac{1}{2}\left({2k-1\choose1}t^{2k}-{2k-1\choose2}t^{2k-1}-\dots+{2k-1\choose 2k-1}t^{2}\right)\\
    =&\frac{t^2}{2}(t^{2k-1}-(t-1)^{2k-1}).
\end{align*}

\end{proof}

\subsection{Chromatic substitution} 

Our second main result is

\begin{theorem}[chromatic substitution]\label{th-cscd}
The value of the coefficient of~$N^0$ (which is the leading term in~$N$) in~$Y(D)(N,2x,x,\dots)$, 
on a chord diagrams $D$
is $2^{|D|}$ times the
chromatic polynomial of the corresponding intersection graph.
\end{theorem}

We call the substitution $p_2=2x, p_4=p_6=\dots=x$ to the prechromatic
substitution the \emph{chromatic substitution} to the $w_\fso$ weight system. 

\begin{example} For the chord diagram~$K_5$, consisting of~$5$
mutually intersecting chords, we have
    \begin{align*} 
        w_\fso(K_5)=&S_2^5-20 S_0 S_2^4+40 S_2^4+140 S_0^2 S_2^3-560 S_0 S_2^3+800 S_2^3-440 S_0^3 S_2^2+2800 S_0^2 S_2^2\\&-7232 S_0 S_2^2+
        +6784 S_2^2+576 S_0^4 S_2-5408 S_0^3 S_2+19712 S_0^2 S_2-32640 S_0 S_2\\&+80 S_0 S_4 S_2-640 S_4 S_2
        +20480 S_2-384 S_0^2 S_4+3840 S_0 S_4-6144 S_4,
\end{align*}
and the substitutions yield
\begin{align*}
        \frac{1}{32}Y(K_5)(N;2x,x,x,\dots)=&x(x-1)(x-2)(x-3)(x-4)\\
        -&x \left(5 x^4-60 x^3+245 x^2-395 x+206\right)N^{-1}\\
        +&x \left(10 x^4-140 x^3+725 x^2-1404 x+814\right)N^{-2}\\
        -&x (10 x^4-160 x^3+1055 x^2-2511 x+1616 )N^{-3}\\
        +&x \left(5 x^4-90 x^3+740 x^2-2140 x+1496\right) N^{-4}\\
        -&x \left(x^4-20 x^3+200 x^2-688 x+512\right)N^{-5}.
    \end{align*}
\end{example}

The notion of intersection graph of a chord diagram can be extended to
arbitrary permutations~\cite{KKL24}. Namely, the \emph{intersection graph}
of a permutation~$\alpha$ is the graph whose set of vertices is
in one-to-one correspondence with the set $V(\alpha)$ of disjoint cycles in~$\alpha$,
and two vertices are connected by an edge iff the corresponding cycles
interlace (that is, the restriction of~$\alpha$ to this pair of cycles
is not the product of its restrictions to each of the cycles).

Theorem~\ref{th-cscd} admits the following extension to intersection
graphs of monotone permutations.

\begin{theorem}[chromatic substitution for permutations]\label{th-csp}
  The value of the coefficient $Y_0(2x,x,x,\dots)$ of~$N^0$, on a monotone permutation~$\alpha$
(which is the leading term in~$N$) is $(-1)^{n(\alpha)}2^{d(\alpha)}$ times the
chromatic polynomial of the corresponding intersection graph, where
$n(\alpha)$ is the number of negative cycles of odd length in~$\alpha$,
and $d(\alpha)$ is the number of length two orbits (`chords') in~$\alpha$.
\end{theorem}

The proof of the theorem is given in the next section.

\subsection{Proof of Theorem~\ref{th-csp}}
The proof of the fact that the leading term in $N$ of the chromatic substitution
on monotone permutations coincides with the chromatic polynomial of their intersection graphs splits into two steps. First, we need to prove that this leading
term depends on the intersection graph only. We show this by proving that
it satisfies the deletion-contraction relation.  In addition, we prove that for a
nonmonotone permutation $\alpha$, the free term in the chromatic substitution is $0$.
The proof proceeds by induction on the number $m$ of permuted elements,
Suppose we have already proved the result for all permutations of at
most $m-1$ elements. For the cyclic positive permutation $\alpha \in \S_m$, the assertions
are also valid, as well as for the negative cyclic permutation.

Let $\alpha\in \S_m$ is a permutation with at least $2$ cycles. Let $v_1$ 
be the cycle in~$\alpha$ containing the element $1$. Let $r+1$ be the smallest number in $v_1$ such that $r$ in not in $v_1$. Denote~$v_2$ the cycle containing $r$.

Apply recurrence (\ref{eq:rec-X1}) to the pair of points $r, r+1$ and then apply the chromatic
substitution to all the resulting terms. 
$$
w_\fso(\alpha)=w_\fso(\alpha')+w_\fso(\beta_1)-w_\fso(\beta_2)+w_\fso(\gamma_1)-w_\fso(\gamma_2).
$$
There are two possibilities:
\begin{itemize}
    \item  the intersection graph of the permutation $\alpha$ and the intersection
graph of the permutation $\alpha'$ are isomorphic. In this case $\beta_1,\beta_2,\gamma_1,\gamma_2$ are permutations of $m-1$ elements containing $c(\alpha)-1$ disjoint cycles.  each of these four permutations contains a nonmonotone
disjoint cycle, the one obtained by merging the cycles $v_1, v_2$ of $\alpha$.
$\beta_1,\beta_2,\gamma_1,\gamma_2$ are not monotone. So the right four terms are $0$.
    \item these two intersection graphs are distinct, the difference between the two
consists in deleting or adding the edge connecting $v_1$ and $v_2$. There is 
exactly one monotone permutation among $\beta_1,\beta_2,\gamma_1,\gamma_2$. 
For example, if $v_1$ and $v_2$ are both positive or negative and 
interlaced in $\alpha'$, but not in $\alpha$, then only $\beta_1$ is monotone,
while $\beta_2,\gamma_1,\gamma_2$ are not. 
Exchanging interlacement in $\alpha'$ and $\alpha$, one switches the roles of the permutations $\beta_1$ to $\beta_2$. And if $v_1$ and $v_2$ have different signs,
then one replaces $\beta$ with $\gamma$. A special case is when
$v_1$ or $v_2$ is both positive and negative,  which means it is a chord. In this case one gets TWO same monotone permutations among $\beta_1,\beta_2,\gamma_1,\gamma_2$.    Hence it becomes the contraction-deletion relation for the chromatic polynomial.
This proves that the coefficient is the chromatic polynomial.

\end{itemize}
In the case the permutation $\alpha$ is not monotone, all the other permutations
 are not monotone as well. Hence, by the induction hypothesis, none of
the other permutations contributes to the coefficient of $N^{m-c(\alpha)}$
in the chromatic substitution, whence the same is true for the permutation $\alpha'$.


\begin{thebibliography}{10}




\bibitem{BN95}
   {D.~Bar-Natan},
  {\it On the Vassiliev knot invariants},
  Topology,
   {1995},
  {\bf 34},
   {423--472}



		
		






\bibitem{CDBook12}
  {S.~Chmutov, S.~Duzhin, and J.~Mostovoy},
  {\it Introduction to Vassiliev Knot Invariants},
  Cambridge University Press,
    {2012}


\bibitem{CL07}
   {S.~Chmutov and S.~Lando},
  {\it Mutant knots and intersection graphs},
  Algebraic \& Geometric Topology,
   {2007},
  {\bf 7},
  {1579--1598}


\bibitem{CDL1}
S.V.Chmutov, S.V.Duzhin, S.K.Lando
{\it
Vassiliev knot invariants I. Introduction}, Advances in Soviet Mathematics, vol. 21 (1994), pp. 117--126

\bibitem{CKL20}
S.~Chmutov, M.~Kazarian, and S.~Lando,
{\it Polynomial graph invariants and the KP
hierarchy},
Selecta Math. New Series
(2020)
26:34






\bibitem{DK07}
S.~Duzhin, M.~Karev,
\emph{Detecting the orientation of string links by finite type invariants}
Functional Analysis and Its Applications, 41(3), 208--216 (2007)




\bibitem{EMM13}
J.~A.~Ellis-Monaghan, I.~Moffatt,
{\it Graphs on surfaces: dualities, polynomials, and knots}, 2013, Springer





\bibitem{KL23} N.~Kodaneva,  S.~Lando, {\it Polynomial graph invariants induced from the
$\fgl$-weight system}, arXiv:2312.17519


\bibitem{KKL24} M.~Kazarian, N.~Kodaneva,  S.~Lando, {\it The universal  $\fgl$-weight system and the chromatic polynomial}, arXiv:2406.10562


\bibitem{KL15} M.~Kazarian, S.~Lando, {\it Combinatorial solutions to integrable hierarchies}, Russ. Math. Surveys, {\bf 70} (3) 453--482 (2015).

\bibitem{KL22} M.~Kazarian, S.~Lando, {\it Weight systems and invariants of graphs and embedded graphs}, Russ. Math. Surveys, {\bf 77} (5) 893--942 (2022).

\bibitem{KY23} M.Kazarian, Zhuoke Yang,
{\it Universal polynomial $\fso$ weight system},
in preparation



\bibitem{KoL23}
N. Kodaneva, S. Lando
\emph{
Polynomial graph invariants induced from the $\fgl$-weight system},
arXiv:2312.17519


\bibitem{K93}
M.~Kontsevich.
\newblock {\it Vassiliev knot invariants},
\newblock in: {\em Advances in Soviet Math.}, 16(2):137--150, 1993.






\bibitem{L97}
  {S.~Lando},
  {\it On primitive elements in the bialgebra of chord diagrams},
   Translations of the American Mathematical Society-Series 2,
  {\bf 180},
   {167--174},
   {1997},
{Providence [etc.] American Mathematical Society, 1949-}




\bibitem{LZ03}
   {S.~Lando and A.~Zvonkin},
  {\it Graphs on Surfaces and Their Applications},
   {2003}




\bibitem{L00}
{J.~Lieberum},
{\it Chromatic weight systems and the corresponding knot invariants},
Math. Ann., Volume 317, pages 459--482 (2000) 








\bibitem{S94} W.~R.~Schmitt,
{\it Incidence Hopf algebras},
		Journal of Pure and Applied Algebra {\bf 96} 299--330 (1994).





\bibitem{V90}
 V.~Vassiliev,
{\it Cohomology of knot spaces}
in: Advance in Soviet Math.
v.~1,
1990,
23--69




\bibitem{ZY22}
Zhuoke Yang,
\newblock {\it New approaches to $\fgl(N)$ weight system}, Izvestiya Mathematics, 2023,
vol.~77:6, 150--166;
arXiv:2202.12225 (2022)

\bibitem{ZY23}
Zhuoke Yang,
\newblock {\it On the Lie superalgebra $\fgl(m|n)$ weight system},
Journal of Geometry and Physics Vol. 187 Article 104808 (2022)















\end{thebibliography}
\end{document}